\documentclass[11pt,reqno]{amsart}

\usepackage[utf8]{inputenc}
\usepackage[T1]{fontenc}
\usepackage[english]{babel}
\usepackage{amsmath,amssymb,amsfonts,amsthm,mathtools,bm,bbm,dsfont,upgreek}
\usepackage[dvipsnames,svgnames,x11names,table]{xcolor}

\usepackage[margin=1.2in]{geometry}
\usepackage{verbatim}
\usepackage{comment}
\usepackage{xurl}
\usepackage{float}

\usepackage[numbers,square,sort]{natbib}
\usepackage{titletoc}
\usepackage{hyperref}
\hypersetup{
    colorlinks=true,
    linkcolor=blue,
    citecolor=Green,
    urlcolor=black,
    filecolor=cyan,
}
\usepackage{aliascnt}
\usepackage[capitalize,noabbrev]{cleveref}
\usepackage{doi}

\title{On Samuels' Conjecture}

\author{Zhi Ling}
\address{National University of Singapore}
\email{lingzhi@nus.edu.sg}

\date{13 Nov 2026}


\theoremstyle{plain}

\newtheorem{theorem}{Theorem}[section]

\newaliascnt{lemma}{theorem}
\newtheorem{lemma}[lemma]{Lemma}
\aliascntresetthe{lemma}

\newaliascnt{proposition}{theorem}
\newtheorem{proposition}[proposition]{Proposition}
\aliascntresetthe{proposition}

\newaliascnt{corollary}{theorem}
\newtheorem{corollary}[corollary]{Corollary}
\aliascntresetthe{corollary}

\newaliascnt{claim}{theorem}
\newtheorem{claim}[claim]{Claim}
\aliascntresetthe{claim}

\theoremstyle{definition}
\newaliascnt{definition}{theorem}
\newtheorem{definition}[definition]{Definition}
\aliascntresetthe{definition}

\theoremstyle{remark}
\newaliascnt{remark}{theorem}
\newtheorem{remark}[remark]{Remark}
\aliascntresetthe{remark}
\newaliascnt{example}{theorem}

\aliascntresetthe{example}

\newcommand{\Var}{\operatorname{Var}}
\newcommand{\Cov}{\operatorname{Cov}}
\newcommand{\dd}{\,\mathrm d}

\begin{document}

\begin{abstract}
Let $0\leq\mu_1\leq\cdots\leq\mu_n$ and let $\lambda>\sum_{i=1}^n\mu_i$. Let $X_1,...,X_n$ be independent nonnegative random variables satisfying
$\mathbb{E}X_i=\mu_i$, and write $D_i := \lambda-\sum_{k=1}^{i-1}\mu_k$ for $1\leq i\leq n$. We prove that
$$
\inf_{X_1,...,X_n} \mathbb{P}\left( \sum_{i=1}^nX_i<\lambda \right)
= \min_{1\leq i\leq n} \prod_{j=i}^n \left( 1-\frac{\mu_j}{D_i} \right).
$$
The bound is sharp and is attained. This proves Samuels' conjecture. Feige's conjecture is thereby resolved, since it follows immediately from the equal-means case. The proof is self-contained.

\end{abstract}

\maketitle

\section{Introduction}
\label{sec.introduction}

Let $n\geq1$, let $0\leq\mu_1\leq\cdots\leq\mu_n$, and let $\lambda>\sum_{i=1}^n\mu_i$. For $\mu\geq0$, denote the set of Borel probability measures on $[0, \infty)$ with mean $\mu$ by $\mathcal{M}(\mu)$. Define
\begin{align}
\label{eq.optimization_domain}
c(\bm{\mu},\lambda)
:=
\inf \mathbb P\left( \sum_{i=1}^n X_i<\lambda \right),
\end{align}
where $\bm{\mu}=(\mu_1,\ldots,\mu_n)$, and the infimum is taken over all independent nonnegative random variables $X_1,...,X_n$ satisfying $\text{Law}(X_i)\in\mathcal{M}(\mu_i)$.

Set
\begin{equation}
\delta := \lambda-\sum_{i=1}^n\mu_i>0,
\quad
D_i := \lambda-\sum_{k=1}^{i-1}\mu_k
= \delta+\sum_{k=i}^n\mu_k,
\quad
q_i(\bm{\mu},\lambda) := \prod_{j=i}^n \left( 1-\frac{\mu_j}{D_i} \right).
\label{eq.threshold_remainders}
\end{equation}
Since $D_i>\mu_j$, for every $j\geq i$, all factors in $q_i(\bm{\mu},\lambda)$ belong to $(0,1]$.

Our main result is \cref{thm.main}.

\begin{theorem}[Samuels' conjecture]
\label{thm.main}
For every integer $n\geq1$, ordered mean vector $0\leq\mu_1\leq\cdots\leq\mu_n$, and $\lambda>\sum_{i=1}^n\mu_i$,
\begin{equation}
c(\bm{\mu},\lambda)
=
\min_{1\leq i\leq n} q_i(\bm{\mu},\lambda)
=
\min_{1\leq i\leq n} \prod_{j=i}^n \left( 1- \frac{\mu_j}{ \lambda-\sum_{k=1}^{i-1}\mu_k } \right).
\label{eq.samuels_formula}
\end{equation}
The bound is sharp and is attained.
\end{theorem}

\begin{remark}
The equivalent upper-tail formulation is
\begin{equation}
\sup \mathbb{P}\left( \sum_{i=1}^nX_i\geq\lambda \right)
=
1- \min_{1\leq i\leq n} \prod_{j=i}^n \left( 1- \frac{\mu_j}{ \lambda-\sum_{k=1}^{i-1}\mu_k } \right).
\end{equation}
\end{remark}

The equal-means result immediately yields the sharp form of Feige's
small-deviation conjectures.

\begin{corollary}[Feige's conjectures]
\label{cor.feige}
Let $X_1,\ldots,X_n$ be independent nonnegative random variables satisfying
$\mathbb{E}X_i\leq1$. Write $S = \sum_{i=1}^nX_i$. For arbitrary $\delta>0$,
\begin{equation}
\inf_{X_1,\ldots,X_n}
\mathbb P\left(S<\mathbb ES+\delta\right)
 =: c_{n,\delta}
 = \min\left\{ \frac{\delta}{1+\delta}, \left( 1-\frac1{n+\delta} \right)^n \right\}.
\label{eq.feige_fixed_dimension}
\end{equation}
The bound is sharp and is attained for every fixed pair $(n,\delta)$.
Moreover,
\begin{equation}
\inf_{n\ge1} c_{n,\delta}
=
\min\left\{ \frac{\delta}{1+\delta}, e^{-1} \right\}.
\label{eq.feige_universal}
\end{equation}
\end{corollary}

\subsection{Background and related work}
\label{subsec.related_work}

This class of problems, in which tail probabilities of random sums are controlled using only independence, nonnegativity, and first-moment information, represents an important paradigm in probability inequalities. The general conjecture was formulated in Samuels' work in the 1960s on Chebyshev- and Markov-type inequalities \cite{Samuels1966,Samuels1968,Samuels1969}. Samuels himself settled the conjecture for $n \leq 4$; for general $n$, he proved the conjecture in the far-tail regime $\lambda\geq (n-1)\sum_i\mu_i$. Beyond these cases, the conjecture remained open in general for several decades.

Feige reintroduced the equal-means version of this conjecture while studying the estimation of the average degree of graphs \cite{Feige2006}. Feige proved $\mathbb{P}\left( S < \mathbb{E}S + \delta \right) \ge \min\left\{ \delta/(1+\delta), 1/13\right\}$ and conjectured that the sharp dimension-free lower bound is $\min\left\{ \delta/(1+\delta), 1/e \right\}$ for every $\delta>0$. The problem now commonly known as \textit{Feige's conjecture} corresponds to the $\delta=1$ case: $\mathbb{P}\!\left( S < \mathbb{E}S + 1 \right) \ge e^{-1}$.
The best proved universal constant was subsequently improved to $1/8$ \citep{HeZhangZhang2010}, $7/50$ \citep{Garnett2020}, and $0.1798$ \citep{GuoHeLingLiu2020}. The $e^{-1}$ bound was also established under additional assumptions, including discrete log-concavity \citep{Alqasem2024} and identically distributed random variables \citep{Egozcue2025}. Paulin showed that Samuels' conjecture implies the dimension-free version of Feige's conjecture \citep{Paulin2017}. Strack and Westermann used dynamic programming to improve the far-tail threshold in the equal-means case of Samuels' conjecture from quadratic order in $n$ to linear \citep{Strack2026}.

Recent advances on Feige's conjecture build on the exact distribution-free mean test of Vlassis and Thomas \citep{Vlassis2026}. Nie and Wei used it to give a short proof for $\delta=1$, obtaining the finite-dimensional constant $(n/(n+1))^n$ \citep{Nie2026}. Fu, Han, Wang, Yan, Zhang, and Zhou combined the same calibration theorem with Grünbaum-type inequalities to establish the sharp finite-dimensional bound for $\delta\ge1$ \citep{Fu2026}. Before the present work, the full sharp conjecture for arbitrary $\delta>0$ remained open.

In a subsequent work, Han and Nie developed an elegant argument based on cap moves and used it to derive a sharp refinement of Samuels' inequality \cite{HanNie2026}. Their bound involves the conditional expectation $\mathbb{E}[ \sum_i X_i \mid \sum_i X_i< \sum_i \mu_i +\delta ]$ and is sharp for any feasible value of this expectation.

\section{Proof of the main theorem}
\label{sec.main_proof}

\subsection{Preliminaries}
In this subsection we state several auxiliary results.

The following inequality on weighted Bernoulli sums is the core of the argument. Its proof is given in \cref{sec.bernoulli_proof}.

\begin{theorem}[Lower-tail bound for weighted Bernoulli sums]
\label{thm.bernoulli}
Let $T := \sum_{i=1}^nw_iB_i$, where $B_i\sim\operatorname{Bernoulli}(p_i)$ independently, $0\leq p_i \leq1$ and $0\leq w_i\leq1$. Let $0<a_1\leq\cdots\leq a_n$ satisfy $w_ip_i\leq a_i$. Assume $d := 1-\mathbb{E}T>0$. For $1\leq i\leq n$, write $C_i := \sum_{j=i}^na_j$.
Then
\begin{equation}
\mathbb{P}(T<1)
\geq
\min_{1\leq i\leq n} \prod_{j=i}^n \left( 1-\frac{a_j}{d+C_i} \right).
\label{eq.bernoulli_bound}
\end{equation}
\end{theorem}

The following results show that every nonnegative distribution can be represented as a mixture of two-point distributions. Thus, a lower bound that holds uniformly for all two-point product laws also holds for the original product law.

\begin{definition}[Mean-one two-point distribution]
\label{def.two_point_law}
For $0\leq x<1<y$, define the two-point distribution with mean one
\begin{equation}
Q_{x,y} := \frac{y-1}{y-x}\delta_x + \frac{1-x}{y-x}\delta_y.
\label{eq.two_point_law}
\end{equation}
To include degenerate distributions, let
$$
\Theta := \{(1,1)\} \sqcup \left\{ (x,y):0\leq x<1<y<\infty \right\},
$$
and set $Q_{1,1}:=\delta_1$.
\end{definition}

\begin{proposition}[Two-point decomposition]
\label{prop.two_point_decomposition}
For every $\nu\in \mathcal{M}(1)$, there exists a probability measure $\eta_\nu$ on $\Theta$ such that
\begin{equation}
\nu(A) = \int_\Theta Q_\theta(A)\,\eta_\nu(\dd\theta).
\label{eq.two_point_mixture}
\end{equation}
$A\subseteq[0,\infty)$ is Borel.
\end{proposition}

\begin{proof}
The mean condition gives
\begin{equation}
c_\nu
:= \int_{[0,1)}(1-x)\,\nu(\dd x)
= \int_{(1,\infty)}(y-1)\,\nu(\dd y).
\label{eq.mean_balance}
\end{equation}
If $c_\nu=0$, then $\nu=\delta_1$, and we take $\eta_\nu=\delta_{(1,1)}$.
Suppose $c_\nu>0$. Define
\begin{equation}
\begin{aligned}
\eta_\nu(\{(1,1)\}) &:= \nu(\{1\}), \\
\eta_\nu(\dd x,\dd y) &:= \frac{y-x}{c_\nu}\, \nu(\dd x)\nu(\dd y), \quad 0\leq x<1<y.
\end{aligned}
\end{equation}
$\eta_\nu$ is a probability measure as
\begin{equation}
\begin{aligned}
\eta_\nu(\Theta)
&= \nu(\{1\}) + \frac1{c_\nu} \int_{[0,1)} \int_{(1,\infty)} (y-x)\,\nu(\dd y)\nu(\dd x)\\
&= \nu(\{1\}) + \nu([0,1)) + \nu((1,\infty)) = 1.
\end{aligned}
\end{equation}
For every bounded nonnegative Borel function $f$,
\begin{equation}
\begin{aligned}
& \int_\Theta \int_{[0,\infty)} f(z)\,Q_\theta(\dd z)\, \eta_\nu(\dd\theta) \\
&= \nu(\{1\})f(1) + \frac1{c_\nu} \int_{[0,1)} \int_{(1,\infty)} \left[ (y-1)f(x)+(1-x)f(y) \right] \nu(\dd y)\nu(\dd x)\\
&= \nu(\{1\})f(1) + \int_{[0,1)}f(x)\,\nu(\dd x) + \int_{(1,\infty)}f(y)\,\nu(\dd y) \\
&= \int_{[0,\infty)}f(z)\,\nu(\dd z).
\end{aligned}
\end{equation}
Taking $f=\mathbf{1}_A$ proves \cref{eq.two_point_mixture}.
\end{proof}

\begin{lemma}[Preservation of a common lower bound]
\label{lem.common_lower_bound}
Let $\nu_i \in \mathcal{M}(1)$ have the representation in \cref{eq.two_point_mixture} with mixing measure $\eta_i$,
and let $A\subseteq[0,\infty)^n$ be Borel. If $q\in[0,1]$ satisfies
\begin{equation}
\left( \bigotimes_{i=1}^nQ_{\theta_i} \right)(A) \geq q
\end{equation}
for $(\eta_1\otimes\cdots\otimes\eta_n)$-almost every $(\theta_1,\ldots,\theta_n)$, then
$$
\left( \bigotimes_{i=1}^n\nu_i \right)(A) \geq q.
$$
\end{lemma}

\begin{proof}
By Tonelli's theorem,
\begin{align}
\left( \bigotimes_{i=1}^n\nu_i \right)(A)
&= \int_{\Theta^n} \left( \bigotimes_{i=1}^nQ_{\theta_i} \right)(A) \prod_{i=1}^n\eta_i(\dd\theta_i) \notag\\
&\geq \int_{\Theta^n} q \prod_{i=1}^n\eta_i(\dd\theta_i) = q.
\end{align}
\end{proof}

\subsection{Bernoullification}
\label{subsec.bernoullification}

In this subsection, we reduce the original problem to the aforementioned Bernoulli theorem and prove the $c(\bm{\mu},\lambda) \geq \min_{1\leq i\leq n} q_i(\bm{\mu},\lambda)$ direction for \cref{thm.main}.

If $\mu_i=0$, then nonnegativity implies $X_i=0$ a.s. Since the means are ordered, all zero means form an initial block. Deleting this block does not change either the event or the minimum in \cref{eq.samuels_formula}. Therefore, it suffices to assume $0<\mu_1\leq\cdots\leq\mu_n$.

Set
\begin{equation}
Z_i := \frac{X_i}{\mu_i}, \quad \nu_i := \text{Law}(Z_i) \in \mathcal{M}(1).
\end{equation}
Apply \cref{prop.two_point_decomposition} to the law of each $Z_i$. By \cref{lem.common_lower_bound}, it suffices to establish a common lower bound for every fixed tuple $(\theta_1,\ldots,\theta_n)\in\Theta^n$.

Let $B_1,...,B_n$ with $B_i\sim\operatorname{Bernoulli}(p_i)$ be independent. Write $\theta_i=(x_i,y_i)$. For a nondegenerate coordinate, let
\begin{equation}
Y_i := x_i+(y_i-x_i)B_i,
\qquad
p_i := \frac{1-x_i}{y_i-x_i}.
\label{eq.two_point_bernoulli}
\end{equation}
For $\theta_i=(1,1)$, set $x_i=y_i=1$, $p_i=0$ (or any fixed value in $[0,1]$). In both cases $Y_i\sim Q_{\theta_i}$.
Define
\begin{equation}
\Lambda := \lambda-\sum_{i=1}^n\mu_ix_i
= \delta+\sum_{i=1}^n\mu_i(1-x_i) >0, \quad
\overline w_i := \frac{\mu_i(y_i-x_i)}{\Lambda}.
\end{equation}
Then
\begin{equation}
\begin{aligned}
\sum_{i=1}^n\mu_iY_i<\lambda
&\Longleftrightarrow
\sum_{i=1}^n
\mu_i(y_i-x_i)B_i
<
\lambda-\sum_{i=1}^n\mu_ix_i \\
&\Longleftrightarrow
\sum_{i=1}^n\overline w_iB_i<1.
\end{aligned}
\end{equation}
Define $w_i := \min\{\overline w_i,1\}$.
The event of interest is unchanged:
\begin{equation}
\mathbf{1}\left\{ \sum_{i=1}^n\overline w_iB_i<1 \right\}
= \mathbf{1}\left\{ \sum_{i=1}^nw_iB_i<1 \right\}.
\end{equation}
This is because, if $B_i=1$ for a coordinate with $\overline w_i\geq1$, both sums are at least one. If all such Bernoulli variables are zero, the two sums coincide.

Let
\begin{equation}
T := \sum_{i=1}^nw_iB_i,
\qquad
a_i := \frac{\mu_i}{\Lambda}>0,
\qquad
d := 1-\mathbb{E}T.
\end{equation}
We verify
\begin{equation}
\begin{aligned}
w_ip_i &\leq \overline w_ip_i
= \frac{\mu_i(1-x_i)}{\Lambda}
\leq \frac{\mu_i}{\Lambda}
= a_i, \quad 1\leq i \leq n
\end{aligned}
\end{equation}
and
\begin{equation}
d = 1-\sum_{i=1}^nw_ip_i
\geq 1- \frac{\sum_{i=1}^n\mu_i(1-x_i)}{\Lambda}
= \frac{\delta}{\Lambda} >0.
\label{eq.slack_bound}
\end{equation}

Thus \cref{thm.bernoulli} applies. Since every factor on the right-hand side of \cref{eq.bernoulli_bound} is increasing in $d$,
\begin{equation}
\begin{aligned}
\mathbb{P}\left( \sum_{i=1}^n\mu_iY_i<\lambda \right)
&=
\mathbb{P}(T<1) \\
&\geq
\min_{1\leq i\leq n} \prod_{j=i}^n \left( 1- \frac{a_j}{ d+\sum_{k=i}^na_k } \right) \\
&\geq
\min_{1\leq i\leq n} \prod_{j=i}^n \left( 1- \frac{\mu_j/\Lambda}{ \delta/\Lambda+\sum_{k=i}^n\mu_k/\Lambda } \right) \\
&=
\min_{1\leq i\leq n} \prod_{j=i}^n \left( 1- \frac{\mu_j}{ \lambda-\sum_{k=1}^{i-1}\mu_k } \right) \\
&= \min_{1\leq i\leq n} q_i(\bm{\mu},\lambda)
\end{aligned}
\end{equation}
Let $A = \left\{ (z_1,...,z_n)\in[0,\infty)^n: \sum_i\mu_i z_i<\lambda \right\}$. Then for fixed $\theta_1,...,\theta_n$,
$$
\left(\bigotimes_{i=1}^nQ_{\theta_i}\right)(A)
 = \mathbb{P}\left( \sum_{i=1}^n\mu_iY_i<\lambda\right)
 \geq \min_{1\leq i\leq n} q_i(\bm{\mu},\lambda).
$$
This bound is independent of the mixing parameters. Therefore \cref{lem.common_lower_bound} transfers it to the original product law.
$$
\begin{aligned}
\left( \bigotimes_{i=1}^n \nu_i \right)(A)
= \mathbb P\left(\sum_i\mu_iZ_i<\lambda\right)
= \mathbb P\left(\sum_iX_i<\lambda\right)
\geq \min_{1\leq i\leq n} q_i(\bm{\mu},\lambda).
\end{aligned}
$$

\subsection{Sharpness}
\label{subsec.sharpness}

Fix $i\in\{1,\ldots,n\}$. Construct $X_1,...,X_n$ such that
\[
X_j :=
\begin{cases}
\mu_j, & j<i,\\[4pt]
D_i B_j, \ B_j\sim\operatorname{Bernoulli}\!\left(\dfrac{\mu_j}{D_i}\right), & j\ge i,
\end{cases}
\]
where $(B_j)_{j\ge i}$ are independent. Since
\begin{equation}
\begin{aligned}
\sum_{j=1}^nX_j<\lambda
&\Longleftrightarrow
\sum_{j=1}^{i-1}\mu_j + D_i\sum_{j=i}^nB_j < \sum_{j=1}^{i-1}\mu_j+D_i \\
&\Longleftrightarrow
B_i=\cdots=B_n=0,
\end{aligned}
\end{equation}
we have
\begin{equation}
\mathbb{P}\left( \sum_{j=1}^nX_j<\lambda \right)
=
\prod_{j=i}^n \left( 1-\frac{\mu_j}{D_i} \right)
=
q_i(\bm{\mu},\lambda).
\label{eq.sharp_probability}
\end{equation}
Every candidate is therefore attained. Taking an index $i$ minimizing $q_i$ proves $c(\bm{\mu},\lambda) \leq \min_{1\leq i\leq n} q_i(\bm{\mu},\lambda)$  and completes the proof of \cref{thm.main}.

\section{Proof of \cref{thm.bernoulli}}
\label{sec.bernoulli_proof}

The core of the proof is a carefully designed Bellman-type induction.

\subsection{Interior parameters}
We first handle interior parameters
\begin{equation}
\label{eq.interior_parameters}
0\leq w_i\leq1,
\qquad
0<p_i<1,
\qquad
1\leq i\leq n.
\end{equation}
Since $p_i<1$, the all-zero Bernoulli state has positive probability, so $\mathbb{P}(T<1)>0$. Define
\begin{equation}
\widetilde p_i := \mathbb{P}(B_i=1\mid T<1), \quad
L := \mathbb ET-\mathbb E[T\mid T<1], \quad
R := 1-\mathbb E[T\mid T<1] = d + L.
\end{equation}
For deletion of coordinate $i$, write
\begin{equation}
T_{-i} :=  T - w_i B_i, \quad
L_{-i} := \mathbb{E}T_{-i} - \mathbb{E}[T_{-i}\mid T_{-i}<1], \quad
R_{-i} := 1- \mathbb{E}[T_{-i}\mid T_{-i}<1].
\end{equation}
Set
\begin{equation}
\psi(x)
:=
\begin{cases}
-\log(1-x)/x,&0<x<1,\\
1,&x=0.
\end{cases}
\end{equation}
We'll use the fact that $\psi$ is monotonically increasing on $[0,1)$.

\begin{claim}
\label{claim.basic_bounds}
Under \cref{eq.interior_parameters},
\begin{equation}
0\leq\widetilde p_i\leq p_i,
\qquad
0\leq L\leq\sum_{i=1}^na_i = C_1,
\qquad
\mathbb{P}(T<1)\geq\frac{d}{R}.
\label{eq.basic_properties}
\end{equation}
\end{claim}

\begin{proof}
Independence of $B_i$ and $T_{-i}$ gives
\begin{equation}
\begin{aligned}
\mathbb{P}(T<1\mid B_i=1)
&=
\mathbb{P}(T_{-i}<1-w_i) \\
&\leq
\mathbb{P}(T_{-i}<1) \\
&=
\mathbb{P}(T<1\mid B_i=0).
\end{aligned}
\end{equation}
Bayes' formula yields
\begin{align}
\widetilde p_i
&=
\frac{
p_i\mathbb{P}(T_{-i}<1-w_i)
}{
p_i\mathbb{P}(T_{-i}<1-w_i)
+
(1-p_i)\mathbb{P}(T_{-i}<1)
}
\leq p_i.
\end{align}
Also, $L = \sum_{i=1}^nw_i(p_i-\widetilde p_i)$, and $0 \leq w_i(p_i-\widetilde p_i) \leq w_ip_i \leq a_i$, which proves $0\leq L\leq\sum_i a_i = C_1$.

Finally,
\begin{equation}
\begin{aligned}
1-d
= \mathbb{E}T
&= \mathbb{E}\left[ T\mathbf{1}_{\{T<1\}} \right] + \mathbb{E}\left[ T\mathbf{1}_{\{T\geq1\}} \right] \\
&\geq \mathbb{P}(T<1) \mathbb{E}[T\mid T<1] + \mathbb{P}(T\geq1) \\
&= \mathbb{P}(T<1)(1-R) + 1-\mathbb{P}(T<1) \\
&= 1-\mathbb{P}(T<1)R.
\end{aligned}
\end{equation}
Hence
$
\mathbb{P}(T<1)\geq d/R
$.
\end{proof}

In particular, if $L\leq a_n$, then
\begin{equation}
\label{eq.small_deficit_case}
\mathbb P(T<1)
\geq \frac d{d+L}
\geq \frac d{d+a_n}
= 1-\frac{a_n}{d+a_n}
= \prod_{j=n}^n \left( 1-\frac{a_j}{d+C_n} \right).
\end{equation}
Therefore, in the following, we only need to consider $L>a_n$, in which case $R=d+L>a_n$.

\begin{claim}
\label{claim.induction_argument}
Assume $0<p_i<1$ for every $i$ and $R>a_n$. Then there exists $0\leq\ell_i\leq a_i,\ 1\leq i\leq n$, such that $\sum_{i=1}^n\ell_i=L$ and
$$
\mathbb P(T<1)
\geq
\prod_{i=1}^n \left( 1-\frac{a_i}{R} \right)^{\ell_i/a_i}.
$$
\end{claim}

\begin{proof}
We argue by induction on $n$. We use the convention that $\max\varnothing=0$, $\prod_{\varnothing}=1$, and $\sum_{\varnothing}=0$.

For $n=0$, we have $T\equiv 0$, so $L=0$, and $\mathbb{P}(T<1)=1$. The claim holds.

For $n\geq 1$, if $w_i=0$ for all $1\leq i \leq n$, then $T\equiv 0$. The claim holds. So we assume at least one $w_i>0$. Since $\Var(T\mid T<1)  = \sum_{i=1}^n w_i\Cov(T,B_i\mid T<1) \geq0$, there exists at least one index $i$ such that $\Cov(T,B_i\mid T<1)\geq0$ and $w_i>0$. Fix such an index in the following.

We have
\begin{equation}
\begin{aligned}
\operatorname{Cov}(T,B_i\mid T<1)
&= \mathbb E[TB_i\mid T<1] - \widetilde p_i\mathbb E[T\mid T<1] \\
&= \mathbb E[T\mid T<1] - (1-\widetilde p_i)\mathbb{E}[T\mid T<1,B_i=0] - \widetilde p_i\mathbb E[T\mid T<1] \\
&= (1-\widetilde p_i) \left( \mathbb{E}[T\mid T<1] - \mathbb{E}[T\mid T<1,B_i=0] \right) \\
&= (1-\widetilde p_i) \left( \mathbb{E}[T\mid T<1] - \mathbb{E}[T_{-i}\mid T_{-i}<1] \right) \\
&= (1-\widetilde p_i)(R_{-i}-R) \geq 0,
\end{aligned}
\end{equation}
which gives
\begin{equation}
R_{-i}\geq R.
\label{eq.room_monotonicity}
\end{equation}

Define
\begin{equation}
\mathcal{A}_i := \{1-w_i\leq T_{-i}<1\},
\quad
\rho_i := \frac{p_i\mathbb{P}(\mathcal{A}_i)}{\mathbb{P}(T_{-i}<1)},
\quad
\Delta_i := L-L_{-i}.
\end{equation}
We have the following exact relationship
\begin{align}
\mathbb{P}(T<1)
&= (1-p_i)\mathbb{P}(T_{-i}<1)
+ p_i\mathbb{P}(T_{-i}<1-w_i)
\notag\\
&= (1-p_i)\mathbb{P}(T_{-i}<1) + p_i \left[ \mathbb{P}(T_{-i}<1) - \mathbb{P}(\mathcal{A}_i) \right]
\notag\\
&= \mathbb{P}(T_{-i}<1) - p_i\mathbb{P}(\mathcal{A}_i)
\notag\\
&= \mathbb{P}(T_{-i}<1)(1-\rho_i).
\label{eq.probability_recursion}
\end{align}
Next, we make some estimates on $\Delta_i$. First
\begin{equation}
\begin{aligned}
& \mathbb{P}(T<1)\mathbb{E}[T\mid T<1] \\
&= \mathbb{E}\left[ T\mathbf{1}_{\{T<1\}} \right] \\
&=
(1-p_i) \mathbb{E}\left[ T_{-i}\mathbf{1}_{\{T_{-i}<1\}} \right]
+
p_i \mathbb{E}\left[ (T_{-i}+w_i) \mathbf{1}_{\{T_{-i}<1-w_i\}} \right] \\
&=
(1-p_i) \mathbb{E}\left[ T_{-i}\mathbf{1}_{\{T_{-i}<1\}} \right]
+
p_i \mathbb{E}\left[ (T_{-i}+w_i) \mathbf{1}_{\{T_{-i}<1\}} \right]
- p_i \mathbb{E}\left[ (T_{-i}+w_i) \mathbf{1}_{\mathcal{A}_i} \right] \\
&=
\mathbb{E}\left[ T_{-i}\mathbf{1}_{\{T_{-i}<1\}} \right]
+
p_iw_i\mathbb{P}(T_{-i}<1)
- p_i \mathbb{E}\left[ (T_{-i}+w_i) \mathbf{1}_{\mathcal{A}_i} \right] \\
&=
\mathbb{P}(T_{-i}<1) \mathbb{E}[T_{-i}\mid T_{-i}<1]
+
p_iw_i\mathbb{P}(T_{-i}<1)
- p_i \mathbb{E}\left[ (T_{-i}+w_i) \mathbf{1}_{\mathcal{A}_i} \right].
\label{eq.truncated_first_moment}
\end{aligned}
\end{equation}
Combining \cref{eq.probability_recursion} and \cref{eq.truncated_first_moment} gives
\begin{equation}
\begin{aligned}
\Delta_i &= L - L_{-i} \\
&= w_ip_i - \left( \mathbb{E}[T\mid T<1] - \mathbb{E}[T_{-i}\mid T_{-i}<1] \right) \\
&= \frac{p_i}{\mathbb{P}(T_{-i}<1)}
\mathbb{E}\left[
\left(
T_{-i}+w_i-\mathbb{E}[T\mid T<1]
\right)
\mathbf{1}_{\mathcal{A}_i}
\right].
\label{eq.deficit_increment_identity}
\end{aligned}
\end{equation}
On $\mathcal{A}_i$, $T_{-i}+w_i\geq1$, so $T_{-i}+w_i-\mathbb{E}[T\mid T<1] \geq 1-\mathbb{E}[T\mid T<1] = R$, and
\begin{align}
\Delta_i
&\geq
\frac{p_iR\mathbb{P}(\mathcal{A}_i)}{\mathbb{P}(T_{-i}<1)}
= \rho_iR.
\label{eq.deficit_probability_comparison}
\end{align}
On the other hand, \cref{eq.room_monotonicity} gives $\Delta_i \leq w_ip_i \leq a_i$. Together,
\begin{equation}
0\leq\rho_i \leq \frac{\Delta_i}{R} \leq \frac{a_i}{R} <1.
\label{eq.single_step_bounds}
\end{equation}
The monotonicity of $\psi$ gives
\begin{equation}
-\log(1-\rho_i)
= \rho_i\psi(\rho_i)
\leq \frac{\Delta_i}{R} \psi\left(\frac{a_i}{R}\right)
= \frac{\Delta_i}{a_i} \log\frac{R}{R-a_i}.
\label{eq.single_step_log_loss}
\end{equation}

Since $R_{-i}\geq R>a_n\geq\max_{j\neq i}a_j$, we can apply the induction hypothesis to the deleted system $T_{-i}$: there exist $0\leq\ell_j\leq a_j$ for $j\neq i$ such that $\sum_{j\neq i}\ell_j=L_{-i}$ and
\begin{equation}
-\log\mathbb{P}(T_{-i}<1)
\leq
\sum_{j\neq i}
\frac{\ell_j}{a_j}
\log\frac{R_{-i}}{R_{-i}-a_j}.
\end{equation}
For fixed $a_j>0$,
\begin{equation}
\frac{\mathrm d}{\mathrm dr}
\left[ \frac1{a_j}\log\frac{r}{r-a_j} \right]
= -\frac1{r(r-a_j)} <0.
\end{equation}
So $R_{-i}\geq R$ implies
\begin{equation}
-\log\mathbb{P}(T_{-i}<1)
\leq
\sum_{j\neq i} \frac{\ell_j}{a_j} \log\frac{R}{R-a_j}.
\end{equation}
Combining this inequality with
\cref{eq.probability_recursion}
and
\cref{eq.single_step_log_loss},
\begin{align}
-\log\mathbb{P}(T<1)
&=
-\log\mathbb{P}(T_{-i}<1)
-
\log(1-\rho_i)
\notag\\
&\leq
\sum_{j\neq i}
\frac{\ell_j}{a_j}
\log\frac{R}{R-a_j}
+
\frac{\Delta_i}{a_i}
\log\frac{R}{R-a_i}.
\label{eq.ordered_loss_coefficients}
\end{align}
Set $\ell_i:=\Delta_i$. Then $0\leq\ell_i\leq a_i$, and $\sum_{j=1}^n\ell_j = \sum_{j\neq i}\ell_j+\ell_i = L_{-i}+\Delta_i = L$. Exponentiating the last inequality yields
$$
\mathbb P(T<1)
\geq
\prod_{j=1}^n \left( 1-\frac{a_j}{R} \right)^{\ell_j/a_j}.
$$
The induction is complete.

\end{proof}

\begin{claim}
\label{claim.interpolation_bound}
Assume $L> a_n$. Since $L\leq C_1$, choose $1\le i\le n-1$ such that $C_{i+1}\leq L\leq C_i$. Then
\begin{equation}
\mathbb{P}(T<1) \geq F_i(L),
\label{eq.interpolation_bound}
\end{equation}
where
\begin{equation}
F_i(L)
:=
\prod_{j=i+1}^n
\left( 1-\frac{a_j}{d+L} \right)
\left( 1-\frac{a_i}{d+L} \right)^{(L-C_{i+1})/a_i}.
\label{eq.interpolation_function}
\end{equation}
\end{claim}

\begin{proof}
We want to maximize the right-hand side of \cref{eq.ordered_loss_coefficients} under the constraints $0\leq\ell_j\leq a_j$ and $\sum_j\ell_j=L$.
Let
\begin{equation}
c_j := \frac1{a_j}\log\frac{R}{R-a_j} = \frac1R \psi\left(\frac{a_j}{R}\right).
\end{equation}
Recall that $a_1\leq\cdots\leq a_n$. By the monotonicity of $\psi$, we have $c_1\leq c_2\leq\cdots\leq c_n$.
Then
\begin{equation}
\label{eq.potential_function}
-\log\mathbb{P}(T<1) \leq \sum_{j=1}^n\ell_jc_j.
\end{equation}
Since
\begin{equation}
\ell_i-(L-C_{i+1})
= \ell_i-\sum_{j=1}^n\ell_j+\sum_{j=i+1}^na_j
= -\sum_{j=1}^{i-1}\ell_j + \sum_{j=i+1}^n(a_j-\ell_j),
\end{equation}
we have
\begin{equation}
\begin{aligned}
&
\sum_{j=1}^n\ell_jc_j
-
\left[ \sum_{j=i+1}^na_jc_j + (L-C_{i+1})c_i \right] \\
&=
\sum_{j=1}^{i-1}\ell_jc_j
+
\left[ \ell_i-(L-C_{i+1}) \right]c_i
+
\sum_{j=i+1}^n(\ell_j-a_j)c_j \\
&=
\sum_{j=1}^{i-1}
\ell_j(c_j-c_i)
+
\sum_{j=i+1}^n
(a_j-\ell_j)(c_i-c_j) \leq 0.
\label{eq.ordered_allocation_bound}
\end{aligned}
\end{equation}
The last inequality follows from $\ell_j\geq0$, $a_j-\ell_j\geq0$, and the fact that $c_1,...,c_n$ are ordered.
Combining \cref{eq.potential_function} and \cref{eq.ordered_allocation_bound},
\begin{align}
-\log\mathbb{P}(T<1)
&\leq
\sum_{j=i+1}^n
\log\frac{R}{R-a_j}
+
\frac{L-C_{i+1}}{a_i}
\log\frac{R}{R-a_i}.
\end{align}
Recall $R=d+L$. Exponentiation gives $\mathbb{P}(T<1) \geq F_i(L)$.

\end{proof}

\begin{claim}
\label{claim.endpoint_interpolation}
For all $1\leq i\leq n-1$ and $L \in [C_{i+1}, C_i]$, the function $F_i(L) \geq \min\{F_i(C_{i+1}),F_i(C_i)\}$.
\end{claim}

\begin{proof}
Write
\begin{equation}
g_i(L):=\log F_i(L)
= \sum_{j=i+1}^n \log\left(1-\frac{a_j}{R}\right)
+ \frac{L-C_{i+1}}{a_i} \log\left(1-\frac{a_i}{R}\right).
\end{equation}
Elementary calculus shows that $g_i''(L)<0$ for every interior stationary point, so $g_i$, and hence $F_i$, has no interior minimum. Since $F_i$ is continuous on $[C_{i+1},C_i]$, its minimum is attained at one of the two endpoints.
\end{proof}

Now we can complete the interior case. For $L\leq a_n$, we've proved the theorem in \cref{eq.small_deficit_case}.
For $L>a_n$, choose $i$ as in \cref{claim.interpolation_bound}. Then \cref{eq.interpolation_bound} and \cref{claim.endpoint_interpolation} give
$$
\begin{aligned}
\mathbb P(T<1)
&\geq
F_i(L) \\[2ex]
&\geq \min\{F_i(C_{i+1}),F_i(C_i)\} \\
&=
\min\left\{
\prod_{j=i+1}^n \left( 1-\frac{a_j}{d+C_{i+1}} \right),
\prod_{j=i}^n \left( 1-\frac{a_j}{d+C_i} \right)
\right\}\\
&\geq
\min_{1\leq k\leq n} \prod_{j=k}^n \left( 1-\frac{a_j}{d+C_k} \right).
\end{aligned}
$$

\subsection{Boundary parameters}
Now we handle the general case $p_i\in [0,1], \ 1\leq i \leq n$.

Let
$
\epsilon_m\downarrow0
$
and define
\begin{equation}
p_i^{(m)}
:=
\begin{cases}
(1-\epsilon_m)p_i,&p_i \in (0,1],\\
\epsilon_m,&p_i=0.
\end{cases}
\label{eq.probability_approximation}
\end{equation}
We have $0<p_i^{(m)}<1$ and $w_ip_i^{(m)} \leq a_i$ for sufficiently large $m$.
Let $T_m := \sum_{i=1}^nw_iB_i^{(m)}$, where $B_i^{(m)} \sim \operatorname{Bernoulli}(p_i^{(m)})$ independently. Set $d_m = 1 - \mathbb{E}T_m \to d >0$.
\begin{align}
\mathbb{P}(T_m<1)
&=
\sum_{b\in\{0,1\}^n}
\mathbf{1}_{\left\{ \sum_iw_ib_i<1 \right\}}
\prod_{i=1}^n
\left(p_i^{(m)}\right)^{b_i}
\left(1-p_i^{(m)}\right)^{1-b_i}
\notag\\
&\longrightarrow
\sum_{b\in\{0,1\}^n}
\mathbf{1}_{\left\{ \sum_iw_ib_i<1 \right\}}
\prod_{i=1}^n
p_i^{b_i}(1-p_i)^{1-b_i}
= \mathbb{P}(T<1).
\label{eq.probability_limit}
\end{align}
Applying the interior result to $T_m$ gives
\begin{equation}
\mathbb{P}(T_m<1)
\geq
\min_{1\leq i\leq n}
\prod_{j=i}^n \left( 1-\frac{a_j}{d_m+C_i} \right).
\end{equation}
Letting $m\to\infty$ yields
\begin{equation}
\mathbb{P}(T<1)
\geq
\min_{1\leq i\leq n}
\prod_{j=i}^n \left( 1-\frac{a_j}{d+C_i} \right).
\end{equation}
This completes the proof of \cref{thm.bernoulli}.

\par\bigskip
\noindent\textbf{AI disclosure.} The author used GPT-5.6 Sol for mathematical discussions. The manuscript was written entirely by the author, who takes full responsibility for its content.

\clearpage
\bibliographystyle{alpha}
\bibliography{ref.bib}

\end{document}